\documentclass[a4paper,11pt,leqno]{smfart}

\usepackage{amssymb,amsmath,enumerate,verbatim,mathrsfs,graphics,graphicx,mathptm,float,mathtools,xcolor,pgf,tikz,epsfig,caption}
\usepackage[T1]{fontenc}
\usepackage[latin1]{inputenc}
\usepackage[english, francais]{babel}
\usepackage[all]{xy}
\usepackage[pdfstartview=FitH,
            colorlinks=true,
            linkcolor=blue,
            urlcolor=blue,
            citecolor=red,
            bookmarks,
            bookmarksopen=true,
            bookmarksnumbered=true]{hyperref}

\usepackage{cleveref}
\theoremstyle{plain}
\newtheorem{theorem}{Th\'eor\`eme}
\newtheorem{theoreme}[theorem]{Th\'eor\`eme}

\theoremstyle{definition}
\newtheorem{definition}{D\'efinition}

\theoremstyle{remark}

\theoremstyle{plain}
\newtheorem{thmsec}{Th\'eor\`eme}[section]

\newtheorem{pro}[thmsec]{Proposition}
\newtheorem{lem}[thmsec]{Lemme}

\theoremstyle{definition}

\theoremstyle{remark}

\def\og{\leavevmode\raise.3ex\hbox{$\scriptscriptstyle\langle\!\langle$~}}
\def\fg{\leavevmode\raise.3ex\hbox{~$\!\scriptscriptstyle\,\rangle\!\rangle$}}

\addtocounter{section}{0}             % Start with section 1
\numberwithin{equation}{section}       % Number formulas within sections

\newcommand{\N}{\mathbb{N}}

\newcommand{\C}{\mathbb{C}}

\newcommand{\pp}{\mathbb{P}^{2}_{\mathbb{C}}}
\newcommand{\pd}{\mathbb{\check{P}}^{2}_{\mathbb{C}}}
\newcommand\pgcd{\mathrm{pgcd}}

\newcommand\Sing{\mathrm{Sing}}
\newcommand\Tang{\mathrm{Tang}}
\newcommand\Leg{\mathrm{Leg}}

\newcommand\radH{\Sigma_{\mathcal{H}}^{\mathrm{rad}}}
\newcommand\radHd{\check{\Sigma}_{\mathcal{H}}^{\mathrm{rad}}}
\newcommand\F{\mathcal{F}}
\newcommand\Hcal{\mathcal{H}}
\newcommand\Ccal{\mathcal{C}}
\newcommand\pref{\mathscr{F}}
\newcommand\preh{\mathscr{H}}
\newcommand\W{\mathcal{W}}

\begin{document}
\title[Platitude des tissus duaux de certains pré-feuilletages convexes]{Platitude des tissus duaux de certains pré-feuilletages convexes du plan projectif complexe}
\date{\today}

\author{Samir \textsc{Bedrouni}}

\address{Facult\'e de Math\'ematiques, USTHB, BP $32$, El-Alia, $16111$ Bab-Ezzouar, Alger, Alg\'erie}
\email{sbedrouni@usthb.dz}


\maketitle{}

\begin{altabstract}\selectlanguage{english}
A holomorphic pre-foliation $\mathscr{F}=\mathcal{C}\boxtimes\mathcal{F}$ on $\mathbb{P}^{2}_{\mathbb{C}}$ is the data of a reduced complex projective curve $\mathcal{C}$ of $\mathbb{P}^{2}_{\mathbb{C}}$ and a holomorphic foliation $\mathcal{F}$ on $\mathbb{P}^{2}_{\mathbb{C}}$. When the foliation $\mathcal{F}$ is convex and the curve $\mathcal{C}$ is invariant by $\mathcal{F}$, we speak of convex pre-foliation. In a previous paper, we showed that if a foliation $\mathcal{F}$ on $\mathbb{P}^{2}_{\mathbb{C}}$ is reduced convex or homogeneous convex and if $\mathcal{C}$ is an invariant line of $\mathcal{F}$, then the dual web of the convex pre-foliation $\mathcal{C}\boxtimes\mathcal{F}$~is~flat. In~this~paper,~we propose to extend this result to the case of a curve $\mathcal{C}$ consisting of several invariant lines.

\end{altabstract}

\selectlanguage{french}
\begin{abstract}
Un pré-feuilletage holomorphe $\mathscr{F}=\mathcal{C}\boxtimes\mathcal{F}$ sur $\mathbb{P}^{2}_{\mathbb{C}}$ est la donnée d'une courbe projective complexe réduite $\mathcal{C}$ de $\mathbb{P}^{2}_{\mathbb{C}}$ et d'un feuilletage holomorphe $\mathcal{F}$ sur $\mathbb{P}^{2}_{\mathbb{C}}.$ Lorsque le feuilletage $\mathcal{F}$ est convexe et que la courbe $\mathcal{C}$ est invariante par $\mathcal{F}$, on parle de pré-feuilletage convexe. Dans un article précédent, nous avons montré que si un feuilletage $\mathcal{F}$ sur $\mathbb{P}^{2}_{\mathbb{C}}$ est convexe réduit ou homogène convexe et si $\mathcal{C}$ est une droite invariante de $\mathcal{F}$, alors le tissu dual du pré-feuilletage convexe $\mathcal{C}\boxtimes\mathcal{F}$ est plat. Nous nous proposons ici d'étendre ce résultat au cas d'une courbe $\mathcal{C}$ composée de plusieurs droites invariantes.
\noindent{\it Classification math\'ematique par sujets (2010). --- 14C21, 32S65, 53A60.}
\end{abstract}

\section{Introduction}
\bigskip

\noindent Cet article poursuit l'étude, commencée dans~\cite{Bed23arxiv}, de la platitude des tissus duaux des pré-feuilletages du plan projectif complexe. Pour les définitions et notations concernant les feuilletages et les tissus de $\pp$, nous renvoyons à~\cite[Sections~1~et~2]{BM18Bull}.

\subsection{Pré-feuilletages convexes du plan projectif complexe}

\noindent Commençons par rappeler la définition d'un pré-feuilletage sur $\pp$.
\begin{definition}[\cite{Bed23arxiv}]
Soient $0\leq k\leq d$ des entiers. Un {\sl pré-feuilletage holomorphe $\mathscr{F}$ sur $\pp$ de co-degré $k$ et de degré $d$}, ou simplement {\sl de type $(k,d)$}, est la donnée d'une courbe projective complexe réduite $\mathcal{C}\subset\pp$ de degré $k$ et d'un feuilletage holomorphe $\F$ sur $\pp$ de degré $d-k.$ On note $\mathscr{F}=\mathcal{C}\boxtimes\F.$ On dit que $\mathcal{F}$ (resp.~$\mathcal{C}$) est le {\sl feuilletage associé} (resp. la {\sl courbe associée}) à $\mathscr{F}.$
\end{definition}

\noindent Les pré-feuilletages de type $(0,d)$ sont précisément les feuilletages de degré $d$ sur $\pp.$
\smallskip

\noindent Rappelons (\emph{voir}~\cite{MP13}) qu'un feuilletage holomorphe sur $\pp$ est dit \textsl{convexe} si ses feuilles qui~ne~sont pas des droites n'ont pas de points d'inflexion. Notons (\emph{voir} \cite{Per01}) que si $\F$ est un feuilletage de degré $d\geq1$ sur $\pp,$ alors $\F$ ne peut avoir plus de $3d$ droites invariantes (distinctes). En outre, si cette borne est atteinte, alors $\F$ est nécessairement convexe; dans ce cas $\F$ est dit \textsl{convexe réduit}. \smallskip

\noindent Dans cet article, nous nous intéressons à une classe particulière de pré-feuilletages, à savoir:
\begin{definition}[\cite{Bed23arxiv}]
Soit $\pref=\mathcal{C}\boxtimes\F$ un pré-feuilletage sur $\pp.$ On dit que $\pref$ est \textsl{convexe} (resp. \textsl{convexe réduit}) si le feuilletage $\F$ est convexe (resp. convexe réduit) et si de plus la courbe $\mathcal{C}$ est invariante par $\F.$
\end{definition}

\noindent Il s'agit d'une extension naturelle des notions de convexité et de convexité réduite des feuilletages.

\subsection{Tissus}

Un $d$-tissu (régulier) $\W$ de $(\mathbb{C}^2,0)$ est la donnée d'une famille $\{\F_1,\F_2,\ldots,\F_d\}$ de feuilletages holomorphes réguliers de $(\mathbb{C}^2,0)$ deux à deux transverses en l'origine. On note $\mathcal{W}=\mathcal{F}_{1}\boxtimes\cdots\boxtimes\mathcal{F}_{d}.$

\noindent Un $d$-tissu (global) sur une surface complexe $M$ est donné dans une carte locale $(x,y)$ par une équation différentielle implicite $F(x,y,y')=0$, où $F(x,y,p)=\sum_{i=0}^{d}a_{i}(x,y)p^{d-i}$ est un polynôme (réduit) en $p$ de degré~$d$, à~coefficients $a_i$ analytiques, avec $a_0$ non identiquement nul. Au voisinage de tout point $z_{0}=(x_{0},y_{0})$ tel que $a_{0}(x_{0},y_{0})\Delta(x_{0},y_{0})\neq 0$, où $\Delta(x,y)$ est le $p$-discriminant de $F$, les courbes intégrales de cette équation définissent un $d$-tissu régulier de $(\mathbb{C}^2,z_{0}).$

\noindent En vertu de~\cite{MP13}, à tout pré-feuilletage $\mathscr{F}=\mathcal{C}\boxtimes\F$ de degré $d\geq1$ et de co-degré $k<d$ sur $\pp$ on~peut associer un $d$-tissu de degré $1$ sur le plan projectif dual $\pd,$ appelé {\sl transformée de \textsc{Legendre}} (ou~tissu dual) de $\pref$ et noté $\Leg\pref$\label{not:Leg-pref}; si $\pref$ est donné dans une carte affine $(x,y)$ de $\pp$ par une $1$-forme $\omega=f(x,y)\left(A(x,y)\mathrm{d}x+B(x,y)\mathrm{d}y\right)$, où $f,A,B\in\mathbb{C}[x,y],$ $\mathrm{pgcd}(A,B)=1,$ alors, dans la carte affine $(p,q)$ de~$\pd$ correspondant à la droite $\{y=px-q\}\subset\pp,$ $\Leg\pref$ est donné par l'équation différentielle implicite
\[
\check{F}(p,q,x):=f(x,px-q)\left(A(x,px-q)+pB(x,px-q)\right)=0, \qquad \text{avec} \qquad x=\frac{\mathrm{d}q}{\mathrm{d}p}.
\]
Lorsque $k\geq1$, $\Leg\pref$ se décompose en $\Leg\pref=\Leg\mathcal{C}\boxtimes\Leg\F$ où $\Leg\mathcal{C}$ est le $k$-tissu algébrique de $\pd$ défini par l'équation $f(x,px-q)=0$ et $\Leg\F$ est le $(d-k)$-tissu irréductible de degré $1$ de $\pd$ décrit par $A(x,px-q)+pB(x,px-q)=0.$

\subsection{Courbure et platitude}

Rappelons brièvement la définition de la courbure d'un~$d$-tissu~avec~$d\geq3.$ Supposons d'abord que $d=3$ et considérons le cas d'un germe de $3$-tissu $\W$ de $(\mathbb{C}^{2},0)$~complètement~décomposable, $\W=\F_1\boxtimes\F_2\boxtimes\F_3.$ Pour $i=1,2,3$, soit $\omega_{i}$ une $1$-forme à singularité isolée en $0$ définissant le~feuilletage~$\mathcal{F}_{i}.$ Sans perdre de généralité, on peut supposer que les $1$-formes $\omega_i$ vérifient $\omega_1+\omega_2+\omega_3=0.$ On montre qu'il existe une $1$-forme méromorphe $\eta(\W)$, bien définie à l'addition près d'une $1$-forme fermée logarithmique $\dfrac{\mathrm{d}g}{g}$ avec $g\in\mathcal{O}^*(\mathbb{C}^{2},0)$, telle que $\mathrm{d}\omega_i=\eta(\W)\wedge\omega_i$ pour $i=1,2,3.$ On définit alors la courbure de $\W$ comme étant la $2$-forme $K(\W)=\mathrm{d}\,\eta(\W).$

\noindent Si maintenant $\W$ est un $d$-tissu complètement décomposable, $\mathcal{W}=\mathcal{F}_{1}\boxtimes\cdots\boxtimes\mathcal{F}_{d}$ avec $d>3$, on définit la courbure $K(\W)$ de $\W$ comme étant la somme des courbures des sous-$3$-tissus de $\W$.

\noindent On peut vérifier que $K(\mathcal{W})$ est une $2$-forme méromorphe à pôles le long du discriminant $\Delta(\mathcal{W})$ de $\mathcal{W},$ canoniquement associée à $\mathcal{W}.$

\noindent Si enfin $\mathcal{W}$ est un $d$-tissu sur une surface complexe $M$ (non forcément complètement décomposable), alors on peut le transformer en un $d$-tissu complètement décomposable au moyen d'un revêtement galoisien ramifié. L'invariance de la courbure de ce nouveau tissu par l'action du groupe de \textsc{Galois} permet de la redescendre en une $2$-forme $K(\W)$ méromorphe globale sur $M,$ à pôles le long du discriminant de $\mathcal{W},$ \emph{voir}~\cite{MP13}.

\noindent Un tissu de courbure nulle est dit plat. Lorsque $M=\pp$ la platitude d'un tissu $\W$ sur $\pp$ se caractérise par l'holomorphie de sa courbure $K(\W)$ le long des points génériques de $\Delta(\W)$.

\subsection{Résultats}

Dans~\cite[Théorème~E]{Bed23arxiv} nous avons montré que le tissu dual d'un pré-feuilletage $\pref=\mathcal{C}\boxtimes\F$ convexe réduit sur $\pp$ de co-degré $1$ ({\it i.e.} dont la courbe $\mathcal{C}$ est une droite invariante de $\F$) est plat. Le~théorème~suivant, démontré au~\S\ref{sec:preuve-theoreme-1}, étend ce résultat au cas de plusieurs droites invariantes.
\begin{theoreme}\label{theoreme:C-invariante-convexe-reduit-plat}
{\sl Soit $\pref=\mathcal{C}\boxtimes\F$ un pré-feuilletage convexe réduit sur $\pp$, dont la courbe associée $\mathcal{C}$ est formée de droites invariantes de $\F$. Alors le tissu $\Leg\pref$ est plat.}
\end{theoreme}

\noindent Suivant~\cite[Définition~B]{Bed23arxiv}, un pré-feuilletage sur $\pp$ est dit {\sl homogène} s'il existe une carte affine $(x,y)$ de $\mathbb{P}^{2}_{\mathbb{C}}$ dans laquelle il est invariant sous l'action du groupe des homothéties complexes $(x,y)\longmapsto \lambda(x,y)$, $\lambda\in \mathbb{C}^{*}.$ Un~pré-feuilletage homogène $\preh$ de type $(k,d)$ sur $\pp$ est alors de la forme $\preh=\mathcal{C}\boxtimes\mathcal{H}$, où $\mathcal{H}$ est un feuilletage homogène de degré $d-k$ sur $\pp$ et où la courbe $\mathcal{C}$ est constituée soit de $k$ droites passant par l'origine $O$, soit de la droite $L_\infty$ et de $(k-1)$ droites passant par $O.$ Le~Corollaire~D~de~\cite{Bed23arxiv} affirme que si $\preh$ est convexe de co-degré $k=1$, alors le $d$-tissu $\Leg\preh$ est plat. Le théorème suivant, prouvé au~\S\ref{sec:preuve-theoreme-2}, généralise~ce~résultat au cas d'un pré-feuilletage homogène convexe quelconque.

\begin{theoreme}\label{theoreme:leg-pre-feuilletage-homogene-convexe-plat}
{\sl Le tissu dual d'un pré-feuilletage homogène convexe sur $\pp$ est plat.}
\end{theoreme}

%--------------------------------------------------------------------------------------------------------

\section{Preuve du Théorème~\ref{theoreme:C-invariante-convexe-reduit-plat}}\label{sec:preuve-theoreme-1}
\bigskip

\noindent La démonstration de ce théorème nécessite quelques résultats intermédiaires.

\begin{lem}\label{lem:K-W-W-prime}
{\sl Soient $\W=\F_1\boxtimes\cdots\boxtimes\F_n$ un $n$-tissu complètement décomposable et $\W'$ un autre tissu. Alors
\begin{align*}
K(\W\boxtimes\W')=K(\W)-(n-2)\sum_{i=1}^{n}K(\F_i\boxtimes\W')+\sum_{1\leq i<j\leq n}K(\F_i\boxtimes\F_j\boxtimes\W')+\binom{n-1}{2}K(\W').
\end{align*}
}
\end{lem}

\begin{proof}[\sl D\'emonstration]
Si $\W'$ est un $d$-tissu, alors on peut écrire localement $\W'=\F^{'}_{1}\boxtimes\cdots\boxtimes\F^{'}_{d}$ et on a
\begin{align*}
&\hspace{-0.5cm}K(\W\boxtimes\W')=K(\W)+\sum_{k=1}^{d}\sum_{1\leq i<j\leq n}K(\F_i\boxtimes\F_j\boxtimes\F^{'}_{k})+\sum_{i=1}^{n}\sum_{1\leq k<k'\leq d}K(\F_i\boxtimes\F^{'}_{k}\boxtimes\F^{'}_{k'})+K(\W').
\end{align*}
Par ailleurs, on a
\begin{align*}
&\hspace{-7.5cm}K(\F_i\boxtimes\W')=K(\W')+\sum_{1\leq k<k'\leq d}K(\F_i\boxtimes\F^{'}_{k}\boxtimes\F^{'}_{k'})
\end{align*}
et
\begin{small}
\begin{align*}
K(\F_i\boxtimes\F_j\boxtimes\W')&=K(\W')+\sum_{k=1}^{d}K(\F_i\boxtimes\F_j\boxtimes\F^{'}_{k})+\sum_{1\leq k<k'\leq d}K(\F_i\boxtimes\F^{'}_{k}\boxtimes\F^{'}_{k'})+\sum_{1\leq k<k'\leq d}K(\F_j\boxtimes\F^{'}_{k}\boxtimes\F^{'}_{k'})\\
&=K(\F_i\boxtimes\W')+K(\F_j\boxtimes\W')-K(\W')+\sum_{k=1}^{d}K(\F_i\boxtimes\F_j\boxtimes\F^{'}_{k}),
\end{align*}
\end{small}
d'où
\begin{align*}
&\hspace{-6.5cm}\sum_{i=1}^{n}\sum_{1\leq k<k'\leq d}K(\F_i\boxtimes\F^{'}_{k}\boxtimes\F^{'}_{k'})=-nK(\W')+\sum_{i=1}^{n}K(\F_i\boxtimes\W')
\end{align*}
et
\begin{Small}
\begin{align*}
\sum_{k=1}^{d}\sum_{1\leq i<j\leq n}K(\F_i\boxtimes\F_j\boxtimes\F^{'}_{k})
&=\sum_{1\leq i<j\leq n}K(\F_i\boxtimes\F_j\boxtimes\W')-\sum_{i=1}^{n}(n-i)K(\F_i\boxtimes\W')-\sum_{j=1}^{n}(j-1)K(\F_j\boxtimes\W')+\binom{n}{2}K(\W')\\
&=\sum_{1\leq i<j\leq n}K(\F_i\boxtimes\F_j\boxtimes\W')-(n-1)\sum_{i=1}^{n}K(\F_i\boxtimes\W')+\binom{n}{2}K(\W').
\end{align*}
\end{Small}

\noindent Par suite,
\begin{align*}
K(\W\boxtimes\W')
=K(\W)+\sum_{1\leq i<j\leq n}K(\F_i\boxtimes\F_j\boxtimes\W')-(n-2)\sum_{i=1}^{n}K(\F_i\boxtimes\W')+\left(\binom{n}{2}-n+1\right)K(\W'),
\end{align*}
d'où l'énoncé.
\end{proof}

\begin{pro}\label{pro:Leg-ell1-ell2-F-plat}
{\sl Soient $d\geq3$ un entier et $\F$ un feuilletage convexe réduit de degré $d-2$ sur $\pp.$ Si~$\ell_1$~et~$\ell_2$~sont des droites $\F$-invariantes distinctes, alors le $d$-tissu $\W_d=\Leg\ell_1\boxtimes\Leg\ell_2\boxtimes\Leg\F$ est plat.
}
\end{pro}

\noindent Pour démontrer cette proposition, nous avons besoin du lemme suivant.

\begin{lem}\label{lem:K-Radial-Radial-F}
{\sl Soit $\W_3=\F_1\boxtimes\F_2\boxtimes\F_3$ un germe de $3$-tissu de $(\C^2,0).$ On suppose que $\F_1$ et $\F_2$ sont des germes de feuilletages radiaux et que $D:=\Tang(\F_1,\F_2)$ est invariante par $\F_3.$ Alors la courbure de $\W_3$ a~au~plus des pôles simples le long de $D.$}
\end{lem}

\begin{proof}[\sl D\'emonstration]
Choisissons un système de coordonnées $(x,y)$ dans lequel $D=\{y=0\}$ et $\F_1$, resp. $\F_2$, resp. $\F_3$, est défini par
\begin{align*}
&\omega_1=\mathrm{d}y,&& \text{resp. } \omega_2=x\mathrm{d}y-y\mathrm{d}x,&& \text{resp. }\omega_3=\mathrm{d}y-y^nh(x,y)\mathrm{d}x,
\end{align*}
avec $n\in\N^*$ et $h(x,0)\not\equiv0.$ Un calcul direct utilisant~\cite[formule~(1.1)]{BM18Bull} montre que la courbure de $\W_3$ s'écrit sous la forme $K(\W_3)=\mathrm{d}\eta(\W_3)$, où
\begin{align*}
\eta(\W_3)=\eta_0+\frac{1}{y}\left(n+1-\frac{h(x,y)+x\partial_xh(x,y)}{h(x,y)(xy^{n-1}h(x,y)-1)}\right)\mathrm{d}y,
\end{align*}
où $\eta_0$ est une $1$-forme holomorphe le long de $D=\{y=0\},$ d'où le lemme.
\end{proof}

\begin{proof}[\sl D\'emonstration de la Proposition~\ref{pro:Leg-ell1-ell2-F-plat}]
Comme $\Leg\ell_1$ et $\Leg\ell_2$ sont des feuilletages radiaux centrés en les points duaux des droites $\ell_1$ et $\ell_2$ respectivement, $\Tang(\Leg\ell_1,\Leg\ell_2)$ est réduit à la droite duale du point $s_0=\ell_1\cap\ell_2$, qui~est~un~point singulier de $\F$ car c'est l'intersection de deux droites $\F$-invariantes. Pour~$i=1,2,$ posons~$\Sigma_{\F}^{\ell_i}:=\Sing\F\cap\ell_i$ et désignons par $\check{\Sigma}_{\F}^{\ell_i}$ l'ensemble des droites duales des points de $\Sigma_{\F}^{\ell_i}$; nous~avons~$\Sigma_{\F}^{\ell_1}\cap\Sigma_{\F}^{\ell_2}=\{s_0\}$, et en vertu de~\cite[Lemme~2.1]{Bed23arxiv}, $\Tang(\Leg\ell_i,\Leg\F)=\check{\Sigma}_{\F}^{\ell_i}$. Par ailleurs, comme $\F$ est convexe réduit, toutes ses singularités sont non-dégénérées (\cite[Lemme~6.8]{BM18Bull}); d'après~\cite[Lemme~2.2]{BFM14} $\Delta(\Leg\F)=\check{\Sigma}_{\F}^{\mathrm{rad}},$ où $\check{\Sigma}_{\F}^{\mathrm{rad}}$ désigne l'ensemble des droites duales des singularités radiales de $\F.$ Il~en~résulte~que
\begin{align*}
\Delta(\W_d)=\check{\Sigma}_{\F}^{\mathrm{rad}}\cup\check{\Sigma}_{\F}^{\ell_1}\cup\check{\Sigma}_{\F}^{\ell_2}.
\end{align*}
Fixons $s\in\Sigma_{\F}^{\mathrm{rad}}\cup\Sigma_{\F}^{\ell_1}\cup\Sigma_{\F}^{\ell_2}$; nous allons décrire le $d$-tissu $\W_d$ près de la droite $\check{s}$ duale de $s.$ Notons $\nu=\tau(\F,s)$ l'ordre de tangence de $\F$ avec une droite générique passant par $s$; alors $s\in\Sigma_{\F}^{\mathrm{rad}}$ si et seulement si $\nu\geq2$, \emph{voir}~\cite[\S1.3]{BM18Bull}. En vertu de la Proposition~3.3~de~\cite{MP13} (\emph{cf.} preuve de~\cite[Théorème~E]{Bed23arxiv}), localement près de $\check{s}$, le $(d-2)$-tissu $\Leg\F$ peut se décomposer en $\Leg\F=\W_{\nu}\boxtimes\W_{d-\nu-2},$ où $\W_{\nu}$ est un $\nu$-tissu irréductible laissant $\check{s}$ invariante et dont la multiplicité du discriminant $\Delta(\W_{\nu})$ le long de $\check{s}$ est minimale, égale à~$\nu-1,$ et où $\W_{d-\nu-2}$ est un $(d-\nu-2)$-tissu régulier et transverse à $\check{s}$. Ainsi, près de la droite $\check{s},$ nous~avons~la~décomposition
\begin{align}\label{equa:W-d}
\W_d=\Leg\ell_1\boxtimes\Leg\ell_2\boxtimes\W_{\nu}\boxtimes\W_{d-\nu-2}.
\end{align}

\noindent Nous allons maintenant étudier l'holomorphie de la courbure de $\W_d$ le long de $\check{s}.$ Il y a trois cas à considérer:
\begin{itemize}
\item [\textbf{\textit{1.}}] Cas où $s\not\in\ell_1\cup\ell_2$; alors $s\in\Sigma_{\F}^{\mathrm{rad}}.$ Dans ce cas, pour $i=1,2$, $\Leg\ell_i$ est transverse à $\check{s}$ et $\check{s}\not\subset\check{\Sigma}_{\F}^{\ell_i}=\Tang(\Leg\ell_i,\Leg\F)$; donc $\W_{d-\nu}:=\Leg\ell_1\boxtimes\Leg\ell_2\boxtimes\W_{d-\nu-2}$ est régulier et transverse à $\check{s}.$ Par suite, la~courbure de $\W_d=\W_{\nu}\boxtimes\W_{d-\nu}$ est holomorphe sur $\check{s}$ par application de \cite[Proposition~2.6]{MP13}.
\smallskip

\item [\textbf{\textit{2.}}] Cas où $s\in\ell_i\setminus\ell_j$ avec $\{i,j\}=\{1,2\}.$ Alors $\check{s}$ est invariante par $\Leg\ell_i$ et $\Leg\ell_j$ est transverse à $\check{s}$; de plus $\check{s}\not\subset\check{\Sigma}_{\F}^{\ell_j}=\Tang(\Leg\ell_j,\Leg\F).$ Ainsi $\W_d=\Leg\ell_i\boxtimes\W_{\nu}\boxtimes\W_{d-\nu-1}$, où $\W_{d-\nu-1}:=\Leg\ell_j\boxtimes\W_{d-\nu-2}$ est régulier et transverse à $\check{s}.$ En appliquant \cite[Théorème~1]{MP13} si $\nu=1$ et \cite[Proposition~3.9]{Bed23arxiv} si $\nu\geq2,$ nous en déduisons que $K(\W_d)$ est holomorphe le long de $\check{s}.$
\smallskip

\item [\textbf{\textit{3.}}] Cas où $s=s_0:=\ell_1\cap\ell_2.$ Alors $\check{s}$ est invariante par $\Leg\ell_1$ et $\Leg\ell_2.$ Nous distinguons deux sous-cas suivant que le point singulier $s$ est radial ou non, {\it i.e.} suivant que $\nu=1$ ou $\nu\geq2.$
\smallskip

\begin{itemize}
\item [$\bullet$] Si $\nu\geq2$ alors la courbure de $\W_d$ est holomorphe sur $\check{s}$ par application de~\cite[Remarque~3.10]{Bed23arxiv}.
\smallskip

\item [$\bullet$] Supposons $\nu=1.$ Alors $\W_d=\Leg\ell_1\boxtimes\Leg\ell_2\boxtimes\W_1\boxtimes\W_{d-3}$; d'après le Lemme~\ref{lem:K-W-W-prime}, nous avons
\begin{align*}
\hspace{1.5cm}K(\W_d)&=K(\Leg\ell_1\boxtimes\Leg\ell_2\boxtimes\W_1)+K(\Leg\ell_1\boxtimes\Leg\ell_2\boxtimes\W_{d-3})-K(\W_1\boxtimes\W_{d-3})+K(\W_{d-3})\\
&\hspace{4.18mm}+\sum_{i=1}^{2}K(\Leg\ell_i\boxtimes\W_1\boxtimes\W_{d-3})-\sum_{i=1}^{2}K(\Leg\ell_i\boxtimes\W_{d-3}).
\end{align*}
Les tissus $\W_{d-3}$, $\W_1\boxtimes\W_{d-3}$ et $\Leg\ell_i\boxtimes\W_{d-3}$ sont réguliers près de $\check{s}$ et donc leurs courbures sont holomorphes sur $\check{s}.$ Les courbures des tissus $\Leg\ell_1\boxtimes\Leg\ell_2\boxtimes\W_{d-3}$ et $\Leg\ell_i\boxtimes\W_1\boxtimes\W_{d-3}$ sont holomorphes le long de $\check{s}$ par application de \cite[Théorème~1]{MP13}. Enfin, par le Lemme~\ref{lem:K-Radial-Radial-F}, $K(\Leg\ell_1\boxtimes\Leg\ell_2\boxtimes\W_1)$ a au plus des pôles simples le long de $\check{s}.$ Nous en déduisons qu'il en est de même~pour~$K(\W_d).$
\end{itemize}
\end{itemize}
\smallskip

\noindent Il s'en suit que la $2$-forme $K(\W_d)$ est holomorphe sur $\pd\setminus\{\check{s}_0\}$ et qu'elle a au plus des pôles simples le long de $\check{s}_0.$ Comme le diviseur canonique de $\pp$ est de degré~$-3$, nous en concluons que $K(\W_d)$ est identiquement nulle.
\end{proof}

\begin{proof}[\sl D\'emonstration du Théorème~\ref{theoreme:C-invariante-convexe-reduit-plat}]
Supposons que la courbe $\mathcal{C}$ soit composée de $n\geq1$ droites $\F$-invariantes distinctes $\ell_1,\ell_2,\ldots,\ell_n$. Comme $\F$ est par hypothèse convexe réduit, le tissu $\Leg\F$, resp. $\Leg\ell_i\boxtimes\Leg\F$, resp.~$\Leg\ell_i\boxtimes\Leg\ell_j\boxtimes\Leg\F$ avec $i\neq j$, est plat en vertu de \cite[Théorème~2]{MP13}, resp.~\cite[Théorème~E]{Bed23arxiv}, resp.~la~Proposition~\ref{pro:Leg-ell1-ell2-F-plat}. Le~Théorème~\ref{theoreme:C-invariante-convexe-reduit-plat} s'obtient alors immédiatement en appliquant le Lemme~\ref{lem:K-W-W-prime} avec $\W=\Leg\mathcal{C}=\Leg\ell_1\boxtimes\Leg\ell_2\boxtimes\cdots\boxtimes\Leg\ell_n$ et $\W'=\Leg\F.$
\end{proof}

%---------------------------------------------------------------------------------------------------------------

\section{Preuve du Théorème~\ref{theoreme:leg-pre-feuilletage-homogene-convexe-plat}}\label{sec:preuve-theoreme-2}

\noindent La démonstration du Théorème~\ref{theoreme:leg-pre-feuilletage-homogene-convexe-plat} utilise le Lemme~\ref{lem:K-W-W-prime} et les deux propositions suivantes.

\begin{pro}\label{pro:Leg-L-infini-ell-H-plat}
{\sl Soient $d\geq3$ un entier, $\mathcal{H}$ un feuilletage homogène convexe de degré $d-2$ sur $\pp$ et~$\ell$~une~droite $\mathcal{H}$-invariante passant par l'origine. Alors le $d$-tissu $\W_d=\Leg\hspace{0.3mm}L_\infty\boxtimes\Leg\ell\boxtimes\Leg\mathcal{H}$ est plat.
}
\end{pro}

\begin{proof}[\sl D\'emonstration]
Nous avons $\W_d=\Leg\hspace{0.3mm}L_\infty\boxtimes\Leg\preh$, où $\preh:=\ell\boxtimes\mathcal{H}$. Si le pré-feuilletage homogène $\preh$ est défini par la $1$-forme
\begin{align*}
\omega=(ax+by)\left(A(x,y)\mathrm{d}x+B(x,y)\mathrm{d}y\right),\quad A,B\in\mathbb{C}[x,y]_{d-2},\hspace{2mm}\pgcd(A,B)=1,
\end{align*}
alors, dans la carte affine $(p,q)$ de $\pd$ associée à la droite $\{y=px-q\}\subset\pp,$ le tissu $\Leg\preh$ est décrit par la forme symétrique
\begin{align*}
\check{\omega}=\Big(a\mathrm{d}q+b(p\mathrm{d}q-q\mathrm{d}p)\Big)\Big(A(\mathrm{d}q,p\mathrm{d}q-q\mathrm{d}p)+pB(\mathrm{d}q,p\mathrm{d}q-q\mathrm{d}p)\Big).
\end{align*}
Par ailleurs, le feuilletage radial $\Leg\hspace{0.3mm}L_\infty$ est défini par la $1$-forme $\check{\omega}_0=\mathrm{d}p$; il est aussi donné par
le champ de vecteurs $\mathrm{X}:=q\frac{\partial }{\partial q}$ qui est une {\sl symétrie} du tissu $\Leg\preh$ au sens où son flot $\phi_t=(p,\mathrm{e}^{t}q)$ préserve $\Leg\preh.$ Comme de plus $\mathrm{X}$ est transverse à $\Leg\preh,$ \cite[Théorème~3.4]{Bed23arxiv} entraîne que $K(\W_d)=K(\Leg\preh).$ Or,~d'après~\cite[Corollaire~D]{Bed23arxiv}, $K(\Leg\preh)\equiv0,$ d'où le résultat.
\end{proof}

\begin{pro}\label{pro:Leg-ell1-ell2-ell-H-plat}
{\sl Soient $d\geq3$ un entier et $\mathcal{H}$ un feuilletage homogène convexe de degré $d-2$ sur $\pp.$ Si~$\ell_1$ et $\ell_2$ sont des droites $\mathcal{H}$-invariantes passant par l'origine, alors le $d$-tissu $\W_d=\Leg\ell_1\boxtimes\Leg\ell_2\boxtimes\Leg\mathcal{H}$ est plat.
}
\end{pro}

\begin{proof}[\sl D\'emonstration]
Pour $i=1,2$, posons $s_i:=L_\infty\cap\ell_i\in\Sing\mathcal{H}$ et notons $\check{s_i}$ (resp. $\check{O}$) la droite duale du point $s_i$ (resp. de l'origine $O$); nous avons $\Tang(\Leg\ell_1,\Leg\ell_2)=\check{O}$ et, en vertu de~\cite[Section~2]{Bed23arxiv}, $\Tang(\Leg\ell_i,\Leg\mathcal{H})=\check{O}\cup\check{s}_i$. En outre, la convexité de $\mathcal{H}$ implique, d'après \cite[Lemme~3.2]{BM18Bull}, que $\Delta(\Leg\mathcal{H})=\radHd\cup\check{O},$ où $\radHd$ désigne l'ensemble des droites duales des singularités radiales de $\mathcal{H}.$ Nous en déduisons que
\begin{align*}
\Delta(\W_d)=\radHd\cup\check{O}\cup\check{s}_1\cup\check{s}_2.
\end{align*}

\noindent Soit $s\in\radH\cup\{s_1,s_2\}$. Notons $\nu-1$ l'ordre de radialité de $s$; en vertu de \cite[Proposition~3.3]{MP13}, au voisinage d'un point générique de la droite $\check{s}$ duale de $s$, nous pouvons décomposer le tissu $\Leg\mathcal{H}$ sous la forme $\Leg\mathcal{H}=\W_{\nu}\boxtimes\W_{d-\nu-2},$ où les tissus $\W_{\nu}$ et $\W_{d-\nu-2}$ sont comme dans la démonstration de la Proposition~\ref{pro:Leg-ell1-ell2-F-plat}. Ainsi~$\W_d$~admet une décomposition de la forme~(\ref{equa:W-d}).
\smallskip

\noindent Maintenant, nous allons montrer que pour tout $s\in\radH\cup\{s_1,s_2\}$, la courbure de $\W_d$ est holomorphe le long de $\check{s}.$ Nous distinguons deux cas:
\smallskip

\begin{itemize}
\item [\textbf{\textit{1.}}] Cas où $s\in\{s_1,s_2\}$. Si $\{i,j\}=\{1,2\}$ et $s=s_i$, alors $\check{s}$ est invariante par $\Leg\ell_i$. De plus, comme $s\not\in\ell_j$ et~$\check{s}\not\subset\Tang(\Leg\ell_j,\Leg\mathcal{H})$, le tissu $\W_{d-\nu-1}:=\Leg\ell_j\boxtimes\W_{d-\nu-2}$ est régulier et transverse à $\check{s},$ et nous avons~$\W_d=\Leg\ell_i\boxtimes\W_{\nu}\boxtimes\W_{d-\nu-1}$. En appliquant \cite[Théorème~1]{MP13} si $\nu=1$ et \cite[Proposition~3.9]{Bed23arxiv} si~$\nu\geq2,$ il vient que $K(\W_d)$ est holomorphe sur $\check{s}.$
\smallskip

\item [\textbf{\textit{2.}}] Cas où $s\not\in\{s_1,s_2\}$; alors $s\in\radH.$ Puisque $\radH\subset\Sing\mathcal{H}\cap L_\infty$, $s\not\in\ell_1\cup\ell_2$ et donc pour $i=1,2$, $\Leg\ell_i$ est transverse à $\check{s}$; comme en outre $\check{s}\not\subset\Tang(\Leg\ell_i,\Leg\mathcal{H})$, le tissu $\W_{d-\nu}:=\Leg\ell_1\boxtimes\Leg\ell_2\boxtimes\W_{d-\nu-2}$ est régulier et transverse à $\check{s}.$ Par conséquent, la~courbure de $\W_d=\W_{\nu}\boxtimes\W_{d-\nu}$ est holomorphe le long de $\check{s}$ par application de \cite[Proposition~2.6]{MP13}.
\end{itemize}
\smallskip

\noindent Il en résulte que la courbure $K(\W_d)$ est holomorphe sur $\pd\setminus\check{O}$, ce qui entraîne, d'après~\cite[Lemme~3.1]{BFM14} (\emph{cf.}~\cite[Lemme~3.8]{Bed23arxiv}), que $K(\W_d)\equiv0.$
\end{proof}

\begin{proof}[\sl D\'emonstration du Théorème~\ref{theoreme:leg-pre-feuilletage-homogene-convexe-plat}]
Soit $\preh=\Ccal\boxtimes\Hcal$ un pré-feuilletage homogène convexe sur $\pp.$ Si~$\Ccal$ est formée de $n\geq1$ droites $\Hcal$-invariantes distinctes $\ell_1,\ell_2,\ldots,\ell_n$, alors ou bien toutes les droites $\ell_i$ passent par l'origine~$O$, ou bien une des droites $\ell_i$ est la droite $L_\infty$ et les $(n-1)$ droites restantes passent~par~$O.$ Comme~$\Hcal$~est~convexe par hypothèse, le tissu $\Leg\Hcal$, resp. $\Leg\ell_i\boxtimes\Leg\Hcal$, resp.~$\Leg\ell_i\boxtimes\Leg\ell_j\boxtimes\Leg\Hcal$ avec~$i\neq j$, est plat en vertu de~\cite[Corollaire~3.4]{BM18Bull}, resp.~\cite[Corollaire~D]{Bed23arxiv}, resp. Propositions~\ref{pro:Leg-L-infini-ell-H-plat} et \ref{pro:Leg-ell1-ell2-ell-H-plat}. La~platitude~de~$\Leg\preh$ s'en déduit aussitôt en appliquant le Lemme~\ref{lem:K-W-W-prime} avec $\W=\Leg\ell_1\boxtimes\Leg\ell_2\boxtimes\cdots\boxtimes\Leg\ell_n$ et $\W'=\Leg\Hcal.$
\end{proof}

%==================================================================================================================================================================

\end{document}